\documentclass[11pt]{article}
\usepackage{amsmath,amssymb,enumerate,amsthm}
\usepackage{color}

\usepackage{hyperref}
\hypersetup{
setpagesize=false,
 bookmarksnumbered=true,%
 bookmarksopen=true,%
 colorlinks=true,%
 linkcolor=blue,
 citecolor=red
}

\makeatletter

\@addtoreset{equation}{section}
\def\@cite#1#2{[{{\bfseries #1}\if@tempswa , #2\fi}]}
\makeatother
\newcommand{\N}{\mathbb{N}}
\newcommand{\ep}{\varepsilon}

\newtheorem{theorem}{Theorem}[section]
\newtheorem{lemma}[theorem]{Lemma}
\newtheorem{proposition}[theorem]{Proposition}
\newtheorem{corollary}[theorem]{Corollary}
\theoremstyle{remark}
\newtheorem{remark}{Remark}[section]
\theoremstyle{definition}

\numberwithin{equation}{section}

\makeatletter
   
   \@addtoreset{equation}{section}
\makeatother

\begin{document}
\title{\bf Singular limit problem of abstract \\ second order evolution equations}
\author{Ryo Ikehata \\{\small Department of Mathematics, Graduate School of Education}\\{\small Hiroshima University}\\ {\small Higashi-Hiroshima 739-8524, Japan}\\ and \\ Motohiro Sobajima\thanks{Corresponding author: msobajima1984@gmail.com} \\{\small Department of Mathematics, Faculty of Science and Technology}\\{\small Tokyo University of Science}\\ {\small Noda 278-8510, Japan}}
\date{}
\maketitle
\begin{abstract}
We consider the singular limit problem in a real Hilbert space for abstract second order evolution equations with a parameter $\ep \in (0,1]$. We first give an alternative proof of the celebrated results due to Kisynski \cite{Kisynski1963} from the viewpoint of the energy method. Next we derive a more precise asymptotic profile as $\ep \to +0$ of the solution itself depending on $\ep$ under rather high regularity assumptions on the initial data.   
\end{abstract}

\section{Introduction}\label{sec:introduction}
\footnote[0]{Keywords and Phrases: Second order; Evolution equations; Hilbert spaces; Singular limit; Initial layer; Energy method.}
\footnote[0]{2010 Mathematics Subject Classification. Primary 34G10, 34K26; Secondary 34K30, 34K25.}
Let $H$ be a real Hilbert space with the inner product $(\cdot,\cdot)$ and the norm $\|\cdot\|$. 
In this paper, 
we consider singular limit problems for 
an abstract linear differential equations of the form 
\begin{equation}\label{ADW}
\begin{cases}
\ep u_\ep''(t)+Au_\ep(t)+u_\ep'(t)=0, \quad t \geq 0,
\\
(u_{\ep},u_{\ep}')(0)=(u_0,u_1), 
\end{cases}
\end{equation}
where $A$ is a nonnegative self-adjoint operator in $H$ endowed with domain $D(A)$ 
and $\ep \in (0,1]$. It should be emphasized that we never assume $A$ to be corecive as is frequently studied. The corecive case is out of scope in this research. 

The main topic of this paper is to discuss the asymptotic profile of the solution $u_\ep(t)$ 
as $\ep \to +0$ in the abstract framework. About a clear answer to the singular limit problem one can cite the following celebrated work due to Kisynski \cite{Kisynski1963} in 1963. 
\begin{theorem}[{Kisynski \cite{Kisynski1963}}]\label{thm:kisynski}
The following assertions hold:
\begin{itemize}
     \setlength{\itemsep}{0pt}
\item[{\rm (i)}] If $(u_0,u_1) \in D(A^{1/2})\times H$, then there exists a positive constant 
$C>0$ such that the unique solution $u_{\ep}(t)$ to problem \eqref{ADW} satisfies
\[
\|u_\ep(t)-e^{-tA}u_0\|\leq C\Big(\ep^{1/2}\|A^{1/2}u_0\|+\ep\|u_1\|\Big).
\]
\item[{\rm (ii)}] 
If $(u_0,u_1)\in D(A)\times H$, then there exists a positive constant 
$C>0$ such that the unique solution $u_{\ep}(t)$ to problem \eqref{ADW} satisfies
\[
\|u_\ep(t)-e^{-tA}u_0\|\leq C\ep \Big(\|Au_0\|+\|u_1\|\Big).
\]
\end{itemize}
\end{theorem}
From the results above one can see that the solution $u_{\ep}(t)$ converges to some solution $v(t)$ (as $\ep \to +0$) of the first order evolution equation with some rate of $\ep$:
\begin{equation}\label{PE}
\begin{cases}
Av(t) + v'(t)=0, \quad t>0, 
\\
v(0) = u_0. 
\end{cases}
\end{equation}
The proof of Theorem \ref{thm:kisynski} is based on the (abstract) spectral analysis, and seems to be rather complicated. Furthermore, it is quite important to deduce the so-called initial layer term as stated in the introduction of the paper \cite{Kisynski1963} from the viewpoint of the number of the initial data between \eqref{ADW}) and \eqref{PE}. 

If one employs an idea introduced by Lions \cite[p. 492, (5.10), (5.11)]{BOOK:Lions}, the initial layer function $\theta_{\ep}(t)$ is defined by a solution $\theta_{\ep}(t)$ to the $\ep$-dependent equation:
\begin{equation}\label{IL}
\begin{cases}
\ep \theta_{\ep}''(t) + \theta_{\ep}'(t) = 0, \quad t \geq 0,
\\
\theta_{\ep}(0) = 0, \quad \theta_{\ep}'(0) = u_{1} + Au_{0}. 
\end{cases}
\end{equation}
With this function $\theta_{\ep}(t)$, the precise singular limit problem is formulated as follows:
\[\Vert u_{\ep}(t) - (v(t) + \theta_{\ep}(t))\Vert \to 0, \quad \Vert u_{\ep}(t) - v(t)\Vert \to 0 \quad (\ep \to +0),\]
with some rate. In this connection, in our case one can find exactly
\[\theta_{\ep}(t) = \ep(1-e^{-\frac{t}{\ep}})(u_{1} + Au_{0}).\]
While, almost 40 years later Ikehata \cite{Ikehata2003} derived the following estimate: 
\[\int_{0}^{\infty}\Vert u_{\ep}(t)-v(t)\Vert^{2}dt = O(\ep) \quad (\ep \to +0)\]
under the assumption on the initial data such that $u_{0} \in D(A)$, and 
\begin{equation}\label{IL0}
u_{1} + Au_{0} = 0.
\end{equation}
This result also provides an answer to our singular limit problems. However, the assumption \eqref{IL0} imposed on the initial data is rather strong, and in this case one has $\theta_{\ep}(t) \equiv 0$. Ikehata \cite{Ikehata2003} employed a device which has its origin in Ikehata-Matsuyama \cite{IkMa2002}, that is, in order to get the estimate the author of \cite{Ikehata2003} uses a function $U(t)$ defined by
\begin{equation}\label{ikehata03-U}
U(t) := \int_{0}^{t}(u_{\ep}(s)-v(s))ds.
\end{equation}   
Then, the function $U(t)$ satisfies 
\[
\begin{cases}
U'(t) + AU(t) = -\ep Au_{0}+\ep u_{\ep}'(t), 
\quad t>0, 
\\
U(0) = 0.
\end{cases}
\]
By applying the multiplier method (energy method) to the equation above, Ikehata obtained such a result. In this connection, the technique of \cite{IkMa2002} is an essential modification of the original idea coming from the celebrated Morawetz method in 1961. A similar concept was employed to a function $Z(t)$ defined by
\begin{equation}\label{ikenishi03}
Z(t) := \int_{0}^{t}\left(u(s)-e^{-sA}(u_{0}+u_{1})\right)ds,
\end{equation}   
where $u(t)$ is a solution to the Cauchy problem: 
\begin{equation}\label{ADW-0}
\begin{cases}
u''(t)+Au(t)+u'(t)=0, \quad t \geq 0,
\\
(u,u')(0)=(u_0,u_1). 
\end{cases}
\end{equation}
The function $Z(t)$ was also effective in the paper Ikehata-Nishihara \cite{IkNi2003} in order to derive (almost optimal) error estimates which show the so-called diffusion phenomena of the solution $u(t)$ to problem \eqref{ADW-0}. Soon after Ikehata \cite{Ikehata2003}, in \cite{ChHa2004} Chill-Haraux derived, above all, sharp point-wise in $t$ estimates such that
\[\Vert u_{\ep}(t)-v(t)\Vert \leq \frac{C\ep}{t}(\Vert u_{0}\Vert + \sqrt{\ep}\Vert u_{0}\Vert_{D(A)}), \quad t \geq 1\] 
under the assumption \eqref{IL0}. Chill-Haraux employed the spectral analysis in the abstract form. It should be emphasized that in both results of \cite{Ikehata2003} and \cite{ChHa2004} the difference $\Vert u_{\ep}(t)-v(t)\Vert$ together with $\Vert u_{\ep}'(t)-v'(t)\Vert$ or $\Vert A^{1/2}(u_{\ep}(t)-v(t))\Vert$ can be estimated. It seems that a most difficult part lies in the estimate for $\Vert u_{\ep}(t)-v(t)\Vert$ because the other factors can be included in the energy itself. 

By the way, observing from the viewpoint of the initial layer term, in Hashimoto-Yamazaki \cite{HaYa2007} and Ghisi-Gobbino \cite{GhGo2012} the authors derived the estimates for the quantities $\Vert u_{\ep}'(t) - (v'(t) + \theta_{\ep}'(t))\Vert$ and $\Vert u_{\ep}(t) - v(t)\Vert$ by using the spectral analysis and the energy method, respectively. 

Our first purpose is to give alternative proofs on the results of Kisynski \cite{Kisynski1963} by a version of the energy method introduced later. The second aim is to derive the following new result, which implies   asymptotic expansions in $\ep$ of the solution $u_{\ep}(t)$.
 
Our main result reads as follows.
\begin{theorem}\label{thm:main}
Assume that $(u_0,u_1)\in D(A^{3/2})\times D(A^{1/2})$ . 
Then there exists a positive constant $C$ such that the unique solution $u_{\ep}(t)$ to problem \eqref{ADW} satisfies
\begin{gather}
\nonumber
\Big\|
	u_\ep(t)-e^{-tA}u_0-\ep \Big(e^{-tA}(Au_{0}+u_{1})-tA^2e^{-tA}u_0-e^{-\frac{t}{\ep}}(Au_{0}+u_{1})\Big)
\Big\|
\\
\label{eq:mainthm}
\leq 
C\ep^{3/2}
\Big(
\|A^{3/2}u_0\|+
\|A^{1/2}(Au_0 + u_1)\|\Big)
\end{gather}
for every $t\geq 0$.
\end{theorem}
\begin{remark}
{\rm
A function corresponding to the initial layer in Theorem \ref{thm:main} 
is slightly different from $\theta_\ep(t)$ in \eqref{IL}. 
Indeed, by Theorem \ref{thm:main} we may find the following estimate
\begin{align*}
\|u_\ep(t)-(e^{-tA}u_0+\theta_\ep(t))\|
&\leq \ep \|(1-e^{-tA})(Au_0+u_1)+tA^2e^{-tA}u_0\|
\\
&\quad+
C\ep^{3/2}
\Big(
\|A^{3/2}u_0\|+
\|A^{1/2}(Au_0 + u_1)\|\Big).
\end{align*}
This estimate is enough to determine the singular limit for $u_\ep$. 
However, from the viewpoint of asymptotic expansions 
(with respect to $\ep$), it is natural to introduce 
the modified function in \eqref{eq:mainthm} as a description of initial layer. 
}
\end{remark}

Moreover, if \eqref{IL0} is satisfied, then 
taking $w_\ep(t)=u_\ep'(t)$, we can understand $w_\ep$ as the solution of 
the problem 
\begin{equation}\label{ADW-deri}
\begin{cases}
\ep w_\ep''(t)+Aw_\ep(t)+w_\ep'(t)=0, \quad t \geq 0,
\\
(w_{\ep},w_{\ep}')(0)=(-Au_0,0).
\end{cases}
\end{equation}
As an application of Theorem \ref{thm:main},
the asymptotic expansion of $u_\ep'(t)$ in $\ep$ can be also clarified. 
\begin{corollary}\label{cor:uep'}
Assume that $u_0\in D(A^{5/2})$ and \eqref{IL0}. 
Then there exists a positive constant $C$ such that the unique solution $u_{\ep}(t)$ to problem \eqref{ADW} satisfies
\begin{gather}
\Big\|
	u_\ep(t)-e^{-tA}u_0+\ep tA^2e^{-tA}u_0
\Big\|
\label{eq:cor1}
\leq 
C\ep^{3/2}\|A^{3/2}u_0\|, 
\\
\Big\|
	u_\ep'(t)
	+Ae^{-tA}u_0-\ep \Big(\big(-tA^2e^{-tA}u_0\big)'+e^{-\frac{t}{\ep}}A^2u_0\Big)
\Big\|
\leq 
C\ep^{3/2}\|A^{5/2}u_0\|
\label{eq:cor2}
\end{gather}
for every $t\geq 0$.
\end{corollary}
\begin{remark}
As explained above, in the case of \eqref{IL0} 
the factor of initial layer does not appear in the first asymptotics of $u_\ep(t)$. 
However, Corollary \ref{cor:uep'} points out that 
the asymptotic expansion for $u_\ep'(t)$ 
has an effect of initial layer $e^{-\frac{t}{\ep}} A^2u_0$.
This kind of phenomenon is not known so far.
\end{remark}
Let us describe our idea to make this paper. A main tool is the classical energy method, 
which is widely applied in the hyperbolic equation field. 
However, a definitive idea is used to obtain several asymptotic estimates 
in terms of $H$-norm for the solution $u_{\ep}(t)$ itself, 
not for $u_{\ep}'(t)$ and $A^{1/2}u_{\ep}(t)$ and so on. 
For this purpose, for example, one of ideas is to use the modified version of 
the function $Z(t)$ defined in \eqref{ikenishi03} (see Lemma \ref{lem:uep-v} below):
\begin{equation}\label{Sobajimafunc1}
\tilde{U}(t) := -\ep J_{\ep}u_{1} + \frac{1}{\ep}\int_{0}^{t}\left(u_{\ep}(s)-e^{-sA}(u_{0}+\ep J_{\ep}u_{1})\right)ds,
\end{equation} 
where $J_{\ep} := (I + \ep A)^{-1}$ is the resolvent of $A$. 
In this paper, this type of new functions play crucial roles to apply 
the classical energy method, 
and a hint concerning how to use the function \eqref{Sobajimafunc1} 
has already been introduced in a quite recent paper written by Sobajima \cite{So_AE}. 
In \cite{So_AE}, the author has applied the previous function $Z(t)$ of \eqref{ikenishi03} to the "damped wave" equation to get
\begin{equation}\label{Sobajimafunc2}
Z''(t) + AZ(t) + Z'(t) = A e^{-tA}(u_0 + u_1),
\quad t\geq 0,
\end{equation}
and by developing the energy method to the equation (\ref{Sobajimafunc2}), 
Sobajima \cite{So_AE} has derived precise asymptotic expansions of 
the solution $u(t)$ to problem \eqref{ADW-0} (as $t \to \infty$) 
in abstract framework which shows a diffusion phenomenon, 
while the authors of \cite{Ikehata2003} and/or \cite{IkNi2003} 
applied the function $Z(t)$ to the "parabolic" equations such that
\[
Z'(t) + AZ(t) = -u'(t), 
\quad t> 0.
\]
This has produced a big difference on the results of the diffusion phenomenon. In any case, an idea of the new function $\tilde{U}(t)$ of \eqref{Sobajimafunc1} of this paper has its origin in \cite{So_AE}. 

This paper is organized as follows. In section \ref{sec:prelim}, we prepare two known results to use in the later section. In section \ref{sec:alt_proof} we give alternative proofs for the results due to Kisynski \cite{Kisynski1963}, 
and in section \ref{sec:asym} the complete proof of Theorem \ref{thm:main} will be given by basing on the energy method.

\section{Preliminaries}\label{sec:prelim}

In this section, one shall prepare useful lemmas which will be frequently used throughout of this paper.

To begin with, the following lemma is one of the maximal regularity 
of the $C_0$-semigroup $\{e^{-tA}\}_{t\geq 0}$ (for the proof, see e.g., {Sobajima \cite{So_AE}}).
\begin{lemma}\label{lem:max-reg}
If $f\in H$, then for every $n\in \N\cup\{0\}$ and $t\geq 0$, 
\begin{equation}\label{eq:lem1-1}
\frac{\|e^{-tA}f\|^2}{2}
+
\frac{2^{n}}{n!}\int_0^t s^{n}\|A^{\frac{n+1}{2}}e^{-sA}f\|^2\,ds=\frac{\|f\|^2}{2}.
\end{equation}
\end{lemma}

We also rely on the following elementary lemma. For this lemma we set
\[ J_\ep = (1+\ep A)^{-1},\]
which implies the resolvent operator of $A$.

\begin{lemma}\label{eq:lem1-1-2}
Let $\ep > 0$. Then it is true that
\[\|A^{1/2}J_\ep f\|^2\leq \frac{1}{\ep}\|f\|^2.\] 
for all $f \in H$.
\end{lemma}
\begin{proof}
Let $f \in H$ and put $w=J_\ep f\in D(A)$. 
Then, by definition $w$ satisfies
\[w+\ep Aw=f.\] 
Making an inner product of this equality by $w$, we can deduce
\begin{align*}
\|w\|^2+\ep\|A^{1/2}w\|^2
&=(w,w)+\ep(A^{1/2}w,A^{1/2}w)
\\
&=(w+\ep A w,w)
\\
&\leq \|f\|\|w\|.
\end{align*}
This provides $\|w\|\leq \|f\|$ 
and then $\ep \|A^{1/2}w\|^2\leq \|f\|^2$
which implies the desired estimate.
\end{proof}

\section{An alternative proof of the result by Kisynski}\label{sec:alt_proof}

In this section, we provide an alternative proof 
of the result by Kisynski (Proposition \ref{thm:kisynski}) 
from the viewpoint of the energy method.
We divide the proof into two cases  
$(u_0,u_1)\in D(A^{1/2})\times H$ and $(u_0,u_1)\in D(A)\times H$, which correspond to each results (i) and (ii) of Theorem 1.1, respectively. 

The first proposition is for the case $(u_0,u_1)\in D(A^{1/2})\times H$. 
\begin{proposition}[{Kisynski \cite{Kisynski1963}}]\label{prop:Kisynski1}
If $(u_0,u_1)\in D(A^{1/2})\times H$, then there exists a positive constant 
$C>0$ such that the unique solution $u_{\ep}(t)$ to problem \eqref{ADW}
 satisfies
\[
\|u_\ep(t)-e^{-tA}u_0\|\leq C\Big(\ep^{1/2}\|A^{1/2}u_0\|+\ep\|u_1\|\Big).
\]
\end{proposition}
\begin{remark}\label{Firstrem}
Under the assumption of Proposition \ref{prop:Kisynski1}, it is known that the problem (\ref{ADW}) has a unique (weak) solution $u_{\ep} \in C([0,\infty);D(A^{1/2})) \cap C^{1}([0,\infty);H)$ for each $\ep \in (0,1]$.
\end{remark}
To begin with, we give a decomposition of $u_\ep$ 
into solutions of the first and the second order differential equations. 

\begin{lemma}\label{lem:uep-v}
Assume that $(u_0,u_1)\in D(A^{1/2})\times H$. 
Let $\ep>0$ and 
let $U_{1\ep}$ be a unique solution of the following problem
\begin{align}\label{eq:aux_U1ep}
\begin{cases}
\ep U_{1\ep}''(t)+AU_{1\ep}(t)+U_{1\ep}'(t)=Ae^{-tA}(u_0+\ep J_\ep u_1), 
\quad t\geq 0,
\\
(U_{1\ep},U_{1\ep}')(0)=(-\ep J_\ep u_1, -J_\ep u_1).
\end{cases}
\end{align}
Then $u_\ep(t)=e^{-tA}(u_0+\ep J_\ep u_1)+\ep U_{1\ep}'(t)$. 
\end{lemma}
\begin{proof}
We consider an auxiliary problem 
\begin{align*}
\begin{cases}
\ep U_{*}''(t)+AU_{*}(t)+U_{*}'(t)=AJ_\ep e^{-tA}(u_0+\ep J_\ep u_1), 
\quad t\geq 0,
\\
(U_{*},U_{*}')(0)=(-\ep J_\ep^2 u_1, -J_\ep^2 u_1).
\end{cases}
\end{align*}
By uniqueness of solutions, one has $U_{*}=J_\ep U_{1\ep}$. Moreover, 
it follows that
\begin{gather*}
U_{*}
\in C^2([0,\infty);H)\cap C^1([0,\infty);D(A^{1/2}))\cap C([0,\infty);D(A)),
\\
A^{1/2}U_{*}
\in C^2([0,\infty);H)\cap C^1([0,\infty);D(A^{1/2}))\cap C([0,\infty);D(A)), 
\end{gather*}
and therefore, one has $AU_{*}\in C^1([0,\infty);H)$.
 
Put $w(t)=J_\ep e^{-tA}(u_0+\ep J_\ep u_1)+\ep U_*'(t)$. 
Then one can check 
\begin{align*}
w(0)&=J_\ep(u_0+\ep J_\ep u_1)+\ep(-J_\ep^2 u_1)
\\
&=J_\ep u_0,
\end{align*}
and 
\begin{align*}
w'(t)
&=-AJ_\ep e^{-tA}(u_0+\ep J_\ep u_1)+U_*''(t)
\\
&=-AU_*(t)-U_*'(t).
\end{align*}
Therefore $w'(0)=-A(-\ep J_\ep^2 u_1)-(-J_\ep^2 u_1)=J_\ep u_1$ and 
\begin{align*}
\ep w''(t)
&=-\ep AU_*'(t)-\ep U_*''(t)
\\
&=-\ep AU_*'(t)-AJ_\ep e^{-tA}(u_0+\ep J_\ep u_1)+AU_*(t)+U_*'(t)
\\
&=-\ep AU_{1\ep}'(t)-AJ_\ep e^{-tA}(u_0+\ep J_\ep u_1)-w'(t)
\\
&=-Aw(t)-w'(t).
\end{align*}
These imply $w(t)=J_\ep u_{\ep}(t)$. Combining this identity and 
\[
w(t)=J_\ep e^{-tA}(u_0+\ep J_\ep u_1)+\ep U_*'(t), \quad U_{*}=J_\ep U_{1\ep},
\]
one can conclude that $u_{\ep}(t)=e^{-tA}(u_0+\ep J_\ep u_1)+\ep U_{1\ep}'(t)$ because of the injectivity of $J_\ep$. 
\end{proof}
We derive an energy estimate for $U_{1\ep}'(t)$ 
which will be translated in terms of $u_{\ep}(t)$. 
\begin{lemma}\label{lem:ee-U1}
Let $(u_0,u_1)\in D(A^{1/2})\times H$. Then it holds that
\begin{gather*}
\ep \|U_{1\ep}'(t)\|^2+\|A^{1/2}U_{1\ep}(t)\|^2
+
\int_0^t\|U_{1\ep}'(s)\|^2\,ds
\leq 
\|A^{1/2}u_0\|^2+3\ep\|u_1\|^2.
\end{gather*}
\end{lemma}
\begin{proof}
Taking the inner product of the both sides of \eqref{eq:aux_U1ep} by $U_{1\ep}'$ one has
\begin{align*}
\frac{d}{dt}
\Big(\ep \|U_{1\ep}'(t)\|^2+\|A^{1/2}U_{1\ep}(t)\|^2\Big)
+2\|U_{1\ep}'(t)\|^2
&=
2(U_{1\ep}'(t),\ep U_{1\ep}''(t)+AU_{1\ep}(t)+U_{1\ep}'(t))
\\
&= 
2(U_{1\ep}'(t), Ae^{-tA}(u_0+\ep J_\ep u_1))
\\
&\leq 
\|U_{1\ep}'(t)\|^2+\|Ae^{-tA}(u_0+\ep J_\ep u_1)\|^2
\end{align*}
and therefore it holds that 
\[
\frac{d}{dt}
\Big(\ep \|U_{1\ep}'(t)\|^2+\|A^{1/2}U_{1\ep}(t)\|^2\Big)
+\|U_{1\ep}'(t)\|^2
\leq 
\|Ae^{-tA}(u_0+\ep J_\ep u_1)\|^2. 
\]
In view of Lemma \ref{lem:max-reg} with $n=0$, 
one also has
\begin{align*}
\int_0^\infty
\|Ae^{-tA}(u_0+\ep J_\ep u_1)\|^2
\,dt
&
\leq \frac{1}{2}\|A^{1/2}(u_0+\ep J_\ep u_1)\|^2
\\
&\leq 
\Big(
\|A^{1/2}u_0\|^2+\ep^2\|A^{1/2}J_\ep u_1\|^2
\Big).
\end{align*}
Noting 
\begin{align*}
\ep \|U_{1\ep}'(0)\|^2+\|A^{1/2}U_{1\ep}(0)\|^2
&=
\ep \|J_\ep u_1\|^2+\ep^2\|A^{1/2}J_\ep u_1\|^2,
\end{align*}
one can obtain the desired inequality because of  Lemma \ref{eq:lem1-1-2}.
\end{proof}

Now we are in a position to prove Proposition \ref{prop:Kisynski1}.\\
\begin{proof}[Proof of Proposition \ref{prop:Kisynski1}]
By virtue of Lemmas \ref{lem:uep-v} and \ref{lem:ee-U1}, one has a series of inequalities
\begin{align*}
\|u_\ep(t)-e^{-tA}u_0\|
&=
\|\ep e^{-tA}J_\ep u_1+\ep U_{1\ep}'(t)\|
\\
&\leq 
\ep\|u_1\|+\ep \|U_{1\ep}'(t)\|
\\
&\leq 
\ep\|u_1\|+\ep^{1/2}\Big(\|A^{1/2}u_0\|^2+3\ep\|u_1\|^2\Big)^{1/2}, 
\end{align*}
which implies the desired estimate.
\end{proof}

\begin{remark}
If we use the $L^2$-estimates of $\|U_{1\ep}'(\cdot)\|^2$ in Lemma \ref{lem:ee-U1}, and Lemmas \ref{eq:lem1-1} and \ref{eq:lem1-1-2} one can get
\begin{align*}
\int_{0}^\infty\|u_\ep(t)-e^{-tA}(u_0+\ep u_1)\|^2\,dt
&=
\int_{0}^\infty\|\ep^2Ae^{-tA}J_\ep u_1-\ep U_{1\ep}'(t)\|^2\,dt
\\
&\leq 
2\ep^4
\int_{0}^\infty\|Ae^{-tA}J_\ep u_1\|^2\,dt
+
2\ep^2
\int_{0}^\infty\|U_{1\ep}'(t)\|^2\,dt
\\
&\leq 
\ep^4\|A^{1/2}J_\ep u_1\|^2
+
2\ep^2\|A^{1/2}u_0\|^2+6\ep^3\|u_1\|^2
\\
&\leq 
2\varepsilon^2\|A^{1/2}u_0\|^2+7\ep^3\|u_1\|^2.
\end{align*}
Since $e^{-tA}u_1\notin L^2(0,\infty;H)$ in general, 
the above estimate cannot be improved 
to the one for $u_\ep(t)-e^{-tA}u_0$. 
On the one hand, if $u_1=A^{1/2}w_1$ for some $w_1\in D(A^{1/2})$, 
from Lemma \ref{eq:lem1-1} one can find the estimate
\[
\int_{0}^\infty\|e^{-tA}u_1\|^2\,dt\leq \frac{\|w_1\|^2}{2},
\]
and hence one can arrive at the following estimate, which is an improvement of \cite{Ikehata2003} with $w_{1} := A^{1/2}(-u_{0})$
\[
\int_{0}^\infty\|u_\ep(t)-e^{-tA}u_0\|^2\,dt
\leq C
\Big(
\ep^2\|w_1\|^2+\ep^2\|A^{1/2}u_0\|^2+\ep^3\|u_1\|^2\Big).
\]
\end{remark}

Next we discuss an alternative proof for (ii) of Theorem \ref{thm:kisynski}. 
The statement is as follows again (see Remark \ref{Firstrem}). 
\begin{proposition}[{Kisynski \cite{Kisynski1963}}]\label{prop:Kisynski2}
If $(u_0,u_1)\in D(A)\times H$, then there exists a positive constant 
$C>0$ such that the unique solution $u_{\ep}(t)$ to problem \eqref{ADW}  satisfies
\[
\|u_\ep(t)-e^{-tA}u_0\|\leq C\ep \Big(\|Au_0\|+\|u_1\|\Big).
\]
\end{proposition}

To prove Proposition \ref{prop:Kisynski2}, 
we divide $u_{\ep}$ into two functions $u_{1\ep}$ and $u_{2\ep}$
which satisfy
\begin{equation}\label{ADW1}
\begin{cases}
\ep u_{1\ep}''(t)+Au_{1\ep}(t)+u_{1\ep}'(t)=0,
\quad t\geq 0,
\\
(u_{1\ep},u_{1\ep}')(0)=(u_0,-Au_0),
\end{cases}
\end{equation}
and 
\begin{equation}\label{ADW2}
\begin{cases}
\ep u_{2\ep}''(t)+Au_{2\ep}(t)+u_{2\ep}'(t)=0,
\quad t\geq 0,
\\
(u_{2\ep},u_{2\ep}')(0)=(0,Au_0+u_1). 
\end{cases}
\end{equation}
The decomposition for $u_{1\ep}$ is as follows. 
\begin{lemma}\label{lem:u2ep-v}
Assume that $u_0\in D(A)$. 
Let $\ep>0$ and let $Z_{\ep}$ be a unique solution of the following problem
\begin{align}\label{eq:aux_Zep}
\begin{cases}
\ep Z_{\ep}''(t)+AZ_{\ep}(t)+Z_{\ep}'(t)=A^{3/2}e^{-tA}u_0, 
\\
(Z_{\ep},Z_{\ep}')(0)=(0,0).
\end{cases}
\end{align}
Then $u_{1\ep}(t)=e^{-tA}u_0+\ep A^{1/2}Z_{\ep}(t)$. 
\end{lemma}
\begin{proof}
We consider an auxiliary problem 
\begin{align}\label{eq:aux_Z*}
\begin{cases}
\ep Z_*''(t)+AZ_*(t)+Z_*'(t)=-A^{3/2}J_\ep e^{-tA}u_0, 
\\
(Z_{\ep},Z_{\ep}')(0)=(0,0).
\end{cases}
\end{align}
Then by uniqueness of solutions, we see $Z_*(t)=J_\ep Z_{\ep}(t)$. 

Put $w(t)=J_\ep e^{-tA}u_0+\ep A^{1/2}Z_*(t)$. Then we have 
\[w(0)=J_\ep u_0, \quad w'(t)=-AJ_\ep e^{-tA}u_0+\ep A^{1/2}Z_*'(t)\]
and
\[w'(0)=-AJ_\ep u_0.\]
Moreover, since $A^{1/2}Z_* \in C^2([0,\infty);H)$, one can deduce
\begin{align*}
\ep w''(t)
&=\ep A^2J_\ep e^{-tA}u_0+\ep^2A^{1/2}Z_*''(t)
\\
&=\ep A^2J_\ep e^{-tA}u_0+\ep A^{1/2}\Big(-A^{3/2}J_\ep e^{-tA}u_0-AZ(t)-Z'(t)\Big)
\\
&=-A\big(\ep A^{1/2}Z(t)\big)+\ep A^{1/2}Z'(t)
\\
&=-Aw(t)-w'(t).
\end{align*}
The above conditions with the uniqueness of solutions provide
\[
J_\ep u_{1\ep}(t)=w(t)=J_\ep e^{-tA}u_0+\ep A^{1/2}Z_*(t)
=J_\ep\Big( e^{-tA}u_0+\ep A^{1/2}Z_\ep(t)\Big).
\]
Since $J_\ep$ is injective, the proof is complete. 
\end{proof}

Next, we deal with the estimate of $A^{1/2}Z_\ep(t)$ 
via the energy method. 
\begin{lemma}\label{lem:ee-Zep}
Assume that $u_0\in D(A)$. Let $Z_\ep$ be a unique solution to 
\eqref{eq:aux_Zep}. Then 
\begin{align*}
\ep \|Z_\ep'(t)\|^2+\|A^{1/2}Z_\ep(t)\|^2
+\int_0^t\|Z_\ep'(s)\|^2\,ds
\leq 
\frac{\|Au_0\|^2}{2}.
\end{align*}
\end{lemma}
\begin{proof}
Taking the inner product of the both sides of \eqref{eq:aux_Zep} by $Z_\ep'(t)$, 
one has 
\begin{align}
\nonumber
\frac{d}{dt}\Big(\ep \|Z_\ep''(t)\|^2+\|A^{1/2}Z_\ep(t)\|^2\Big)
+
2\|Z_\ep'(t)\|^2
&=
2(Z_\ep'(t), \ep Z_\ep''(t)+AZ_\ep(t)+Z_\ep'(t))
\\
\nonumber
&=
-2(Z_\ep'(t),A^{3/2}e^{-tA}u_0)
\\
\label{eq:Zest}
&\leq 
\|Z_\ep'(t)\|^2
+
\|A^{3/2}e^{-tA}u_0\|^2.
\end{align}
In view of Lemma \ref{lem:max-reg} with $n=0$, we see that 
\[
\int_0^\infty\|A^{3/2}e^{-tA}u_0\|^2\,dt\leq \frac{\|Au_0\|^2}{2}
\]
Therefore integrating \eqref{eq:Zest} over $[0,t]$, we obtain 
the desired inequality. 
\end{proof}

Summarizing the above lemmas, one can provide a proof of Proposition \ref{prop:Kisynski2}.\\
\begin{proof}[Proof of Proposition \ref{prop:Kisynski2}]
By definition of $u_{1\ep}$ and $u_{2\ep}$ with Lemma \ref{lem:u2ep-v}, 
we already have
\begin{align*}
u_\ep(t)
&=u_{1\ep}(t)+u_{2\ep}(t)
\\
&=e^{-tA}u_0+\ep A^{1/2}Z_{\ep}(t)+u_{2\ep}(t). 
\end{align*}
Applying Proposition \ref{prop:Kisynski1} to $u_{2\ep}$ we can deduce
\[
\|u_{2\ep}(t)\|\leq C\ep \|Au_0+u_1\|.
\]
On the other hand, we see from Lemma \ref{lem:ee-Zep} that 
\[
\|A^{1/2}Z_\ep(t)\|\leq \frac{\|Au_0\|}{\sqrt{2}}.
\]
Consequently, we obtain the desired estimate $\|u_\ep(t)-e^{-tA}u_0\|\leq C\ep (\|Au_0\|+\|u_1\|)$.
\end{proof}

\section{Second order asymptotics in the singular limit}\label{sec:asym}

In this section, let us prove Theorem \ref{thm:main}. 

The following lemma is a consequence of Lemma \ref{lem:uep-v} to problems \eqref{ADW1} and \eqref{ADW2}.
\begin{lemma}\label{lem:u1ep-etAu0}
Assume that $(u_0,u_1)\in D(A^{3/2})\times D(A^{1/2})$. 
Let $u_{1\ep}$ and $u_{2\ep}$ be solutions of the problems 
\eqref{ADW1} and \eqref{ADW2}, respectively. 
Then the following assertions hold.
\begin{itemize}
\item[\bf (i)]
Let $\widetilde{U}_{1\ep}$ be a unique solution of the following problem
\begin{align}\label{eq:aux_U1ep-4}
&\begin{cases}
\ep \widetilde{U}_{1\ep}''(t)+A\widetilde{U}_{1\ep}(t)+\widetilde{U}_{1\ep}'(t)=Ae^{-tA}J_\ep u_0, 
\quad t\geq 0,
\\
(\widetilde{U}_{1\ep},\widetilde{U}_{1\ep}')(0)=(\ep AJ_\ep u_0, AJ_\ep u_0).
\end{cases}
\end{align}
Then $u_{1\ep}(t)=e^{-tA}J_\ep u_0+\ep \widetilde{U}_{1\ep}'(t)$. 
\item[\bf (ii)]
Let $\widetilde{U}_{2\ep}$ be a unique solution of the following problem
\begin{align}\label{eq:aux_U2ep-4}
\begin{cases}
\ep \widetilde{U}_{2\ep}''(t)+A\widetilde{U}_{2\ep}(t)+\widetilde{U}_{2\ep}'(t)=\ep Ae^{-tA}J_\ep v_1, 
\quad t\geq 0,
\\
(\widetilde{U}_{2\ep},\widetilde{U}_{2\ep}')(0)=(-\ep J_\ep v_1, -J_\ep v_1)
\end{cases}
\end{align}
with $v_1=Au_0+u_1$.
Then $u_{2\ep}(t)=\ep e^{-tA}J_\ep v_1+\ep \widetilde{U}_{2\ep}'(t)$. 
\end{itemize}
\end{lemma}
\begin{remark}\label{Firstremof4}
Under the assumption of Theorem \ref{thm:main}, it is known that the problem (\ref{ADW}) has a unique (strong) solution $u_{\ep} \in C([0,\infty);D(A)) \cap C^{1}([0,\infty);D(A^{1/2})) \cap C^{2}([0,\infty);H)$ for each $\ep \in (0,1]$.
\end{remark}

Next we define the respective profile functions of \eqref{eq:aux_U1ep-4} and \eqref{eq:aux_U2ep-4} as 
\begin{align*}
V_{1\ep}(t)
&=2\ep Ae^{-tA}J_\ep u_0+tAe^{-tA}J_\ep u_0, 
\\
V_{2\ep}(t)
&=
-2\ep e^{-tA}J_\ep v_1+\ep tAe^{-tA}J_\ep v_1.
\end{align*}
Then for $j=1,2$, we decompose $\widetilde{U}_{j\ep}$ 
into the semigroup part $V_{j\ep}(t)$ and the other part as a 
solution of second order differential equations.
\begin{lemma}\label{lem:W=Uep-v}
Assume that $(u_0,u_1)\in D(A^{3/2})\times D(A^{1/2})$. 
Let $\widetilde{U}_{1\ep}$ and $\widetilde{U}_{2\ep}$ be solutions of the problems 
\eqref{eq:aux_U1ep-4} and \eqref{eq:aux_U2ep-4}, respectively. Then the following assertions hold.
\begin{itemize}
\item[\bf (i)]
Let $W_{1\ep}$ be a unique solution to the Cauchy problem
\begin{align}\label{eq:aux_W1ep}
\begin{cases}
\ep W_{1\ep}''(t)+AW_{1\ep}(t)+W_{1\ep}'(t)=-V_{1\ep}''(t), 
\quad t\geq 0,
\\
(W_{1\ep},W_{1\ep}')(0)=(-AJ_\ep u_0,-2A^2J_\ep u_0). 
\end{cases}
\end{align}
Then $\widetilde{U}_{1\ep}(t)=V_{1\ep}(t)+\ep W_{1\ep}(t)$. 
\item[\bf (ii)]
Let $W_{2\ep}$ be a unique solution to the Cauchy problem
\begin{align}\label{eq:aux_W2ep}
\begin{cases}
\ep W_{2\ep}''(t)+AW_{2\ep}(t)+W_{2\ep}'(t)=-V_{2\ep}''(t), 
\quad t\geq 0,
\\
(W_{2\ep},W_{2\ep}')(0)=(J_\ep v_1, -2AJ_\ep v_1). 
\end{cases}
\end{align}
Then $\widetilde{U}_{2\ep}(t)=V_{2\ep}(t)-u_{2\ep}(t)+\ep W_{2\ep}(t)$. 
\end{itemize}
\end{lemma}
\begin{proof}
Both parts of the proof of Lemma \ref{lem:W=Uep-v}, {\bf(i)} and {\bf (ii)} are similar to Lemma \ref{lem:u2ep-v}.
\end{proof}
\begin{remark}
The function $\widetilde{U}_{2\ep}$ contains 
the solution of original problem \eqref{ADW2}.
This fact can be understood as the effect of initial layer function as $\ep \to0$. 
\end{remark}

The energy estimates for $W_{j\ep}$ $(j=1,2)$ 
can be seen as follows.
\begin{lemma}\label{lem:ee-W1ep}
Assume that $(u_0,u_1)\in D(A^{3/2})\times D(A^{1/2})$. 
Let $W_{1\ep}$ and $W_{2\ep}$ be unique solutions to problems \eqref{eq:aux_W1ep} and \eqref{eq:aux_W2ep}, respectively. 
Then 
there exist positive constants $C_1>0$ and $C_2>0$ such that
\begin{align*}
\ep \|W_{1\ep}'(t)\|^2+\|A^{1/2}W_{1\ep}(t)\|^2
+\int_0^t\|W_{1\ep}'(s)\|^2\,ds
\leq 
C_1\|A^{3/2}u_0\|^2,
\\
\ep \|W_{2\ep}'(t)\|^2+\|A^{1/2}W_{2\ep}(t)\|^2
+\int_0^t\|W_{2\ep}'(s)\|^2\,ds
\leq 
C_2\|A^{1/2}v_1\|^2.
\end{align*}
\end{lemma}
\begin{proof}
For $j=1,2$, 
taking the inner product of the both sides of (\ref{eq:aux_W1ep}) or (\ref{eq:aux_W2ep}) by $W_{j\ep}'(t)$, 
we have 
\begin{align}
\nonumber
\frac{d}{dt}\Big(\ep \|W_{j\ep}'(t)\|^2+\|A^{1/2}W_{j\ep}(t)\|^2\Big)
+2\|\widetilde{W}_{j\ep}'(t)\|^2
&=
2(W_{j\ep}'(t), \ep W_{j\ep}''(t)+AW_{j\ep}(t)+W'_{j\ep}(t))
\\
\nonumber
&=
-2(\widetilde{W}_{j\ep}'(t),V_{j\ep}''(t))
\\
\label{eq:Wjep-est}
&\leq 
\|\widetilde{W}_{j\ep}'(t)\|^2+
\|V_{j\ep}''(t)\|^2.
\end{align}
Observe that 
\begin{align*}
V_{1\ep}''(t)
&=
2A^2e^{-tA}\Big(\ep A J_\ep u_0-J_\ep u_0\Big)
+
tA^3e^{-tA}J_\ep u_0,
\\
V_{2\ep}''(t)
&=
-4\ep A^2 e^{-tA}J_\ep v_1+\ep tA^3e^{-tA}J_\ep v_1.
\end{align*}
In view of Lemma \ref{lem:max-reg} with $n=0$ and $n=2$, we see that 
\begin{align*}
\int_0^\infty\|V_{1\ep}''(t)\|^2\,dt
&\leq 
8\int_0^\infty\Big\|A^2e^{-tA}\Big(\ep A J_\ep u_0-J_\ep u_0\Big)
\Big\|^2\,dt
+
2\int_0^\infty t^2\|A^3e^{-tA}J_\ep u_0\|^2\,dt
\\
&\leq 
4\Big\|A^{3/2}\Big(\ep A J_\ep u_0-J_\ep u_0\Big)\Big\|^2
+
\frac{1}{2}\|A^{3/2}J_\ep u_0\|^2
\\
&\leq 
C_1'\|A^{3/2}u_0\|^2,
\end{align*}
and also 
\begin{align*}
\int_0^\infty\|V_{2\ep}''(t)\|^2\,dt
&\leq 
32\ep^2\int_0^\infty\|A^2 e^{-tA}J_\ep v_1\|^2\,dt
+
2\ep^2\int_0^\infty t^2\|A^3e^{-tA}J_\ep v_1\|^2\,dt
\\
&\leq 
C_2'\ep^2\|A^{3/2}J_\ep v_1\|^2
\\
&\leq 
C_2'\|A^{1/2}v_1\|^2
\end{align*}
with some positive constants $C_1',C_2'>0$. 
Therefore integrating \eqref{eq:Wjep-est} $(j=1,2)$ over $[0,t]$, and using Lemma \ref{eq:lem1-1-2}
one obtains
\begin{align*}
\ep \|W_{1\ep}'(t)\|^2+\|A^{1/2}W_{1\ep}(t)\|^2
+
\int_0^t\|W_{1\ep}'(s)\|^2\,ds
&\leq 
\ep \|W_{1\ep}'(0)\|^2+\|A^{1/2}W_{1\ep}(0)\|^2
+
C_1'\|A^{3/2}u_0\|^2
\\
&=
\ep \|A^2J_\ep u_0\|^2+2\|A^{3/2}J_\ep u_0\|^2
+
C_1'\|A^{3/2}u_0\|^2
\\
&\leq 
(3+C_1')\|A^{3/2}u_0\|^2,
\end{align*}
and 
\begin{align*}
\ep \|W_{2\ep}'(t)\|^2+\|A^{1/2}W_{2\ep}(t)\|^2
+\int_0^t\|W_{2\ep}'(s)\|^2\,ds
&\leq 
\ep \|W_{2\ep}'(0)\|^2+\|A^{1/2}W_{2\ep}(0)\|^2
+
C_2'\|A^{1/2}v_1\|^2
\\
&=
4\ep \|AJ_\ep v_1\|^2+\|A^{1/2}J_\ep v_1\|^2
+
C_2'\|A^{1/2}v_1\|^2
\\
&\leq 
(5+C_2')\|A^{1/2}v_1\|^2.
\end{align*}
Setting $C_1=3+C_1'$ and $C_2=5+C_2'$,  we obtain the desired inequalities.
\end{proof}

Finally, we prove Theorem \ref{thm:main} via 
the energy estimates for $W_{j\ep}$ $(j=1,2)$.\\ 
\begin{proof}[Proof of Theorem \ref{thm:main}]
By using Lemmas \ref{lem:u1ep-etAu0} and \ref{lem:W=Uep-v}, 
we have
\begin{align*}
u_{1\ep}(t)
&=e^{-tA}J_\ep u_0+\ep \widetilde{U}_{1\ep}'(t)
\\
&=e^{-tA}J_\ep u_0+\ep V_{1\ep}'(t)+\ep^2\widetilde{W}_{1\ep}'(t),
\end{align*}
and 
\begin{align}
\nonumber
u_{2\ep}(t)
&=\ep e^{-tA}J_\ep v_1+\ep \widetilde{U}_{2\ep}'(t)
\\
\label{eq:decom-u2ep}
&=
\ep e^{-tA}J_\ep v_1+\ep V_{2\ep}'(t)-\ep u_{2\ep}'(t)+\ep^2 W_{2\ep}'(t).
\end{align}
Observe that the definition of $V_{1\ep}$
and the relation $J_\ep u_0=u_0-\ep A J_\ep u_0$ 
yield
\begin{align}
\nonumber
e^{-tA}J_\ep u_0+\ep V_{1\ep}'(t)
&=
e^{-tA}J_\ep u_0+\ep \Big(Ae^{-tA}J_\ep u_0-2\ep A^2e^{-tA}J_\ep u_0
-tA^2e^{-tA}J_\ep u_0\Big)
\\
\nonumber
&=
e^{-tA}u_0+\ep \Big(-2\ep A^2e^{-tA}J_\ep u_0
-tA^2e^{-tA}J_\ep u_0\Big)
\\
\label{eq:profile-u2ep}
&=
e^{-tA}u_0
-\ep tA^2e^{-tA}u_0
+\ep^2 \Big(-2A^2e^{-tA}J_\ep u_0
+tA^3e^{-tA}J_\ep u_0\Big).
\end{align}
Therefore 
from Lemma \ref{lem:ee-W1ep}, the uniform boundedness of $\|tAe^{-tA}\|$ and Lemma \ref{eq:lem1-1-2} 
it follows that 
\begin{align}\label{asymptotics1}
\|u_{1\ep}(t)-e^{-tA}u_0+\ep tA^2e^{-tA}u_0\|\leq C\ep^{3/2}
\|A^{3/2}u_0\|.
\end{align}
It remains to find a profile of $u_{2\ep}$. 
Note that 
\begin{align*}
\ep e^{-tA}J_\ep v_1+\ep V_{2\ep}'(t)
&=
\ep e^{-tA}J_\ep v_1+
\ep \Big(\ep Ae^{-tA}J_\ep v_1+2\ep Ae^{-tA}J_\ep v_1-\ep tA^2e^{-tA}J_\ep v_1
\Big)
\\
&=
\ep e^{-tA}v_1+
\ep^2 \Big(2Ae^{-tA}J_\ep v_1-tA^2e^{-tA}J_\ep v_1
\Big).
\end{align*}
Here we already have the following identity via \eqref{eq:decom-u2ep} and \eqref{eq:profile-u2ep}:
\begin{equation}\label{eq:inilayer}
\ep u_{2\ep}'(t)+u_{2\ep}(t)=\ep e^{-tA}v_1+\ep^{3/2}\widetilde{V}_{2\ep}(t),
\end{equation}
where 
\[
\widetilde{V}_{2\ep}(t)=\ep^{1/2}W_{2\ep}'(t)+
\ep^{1/2}\Big(2Ae^{-tA}J_\ep v_1-tA^2e^{-tA}J_\ep v_1
\Big).
\]
In view of Lemma \ref{lem:ee-W1ep} 
and similarly to the estimate for \eqref{asymptotics1} 
one can get 
\begin{align}\label{eq:u2ep-error}
\|\widetilde{V}_{2\ep}(t)\|\leq C'\|A^{1/2}v_1\|.
\end{align}
We see from \eqref{eq:inilayer} that 
\[
\frac{d}{dt}\Big(e^{t/\ep}u_{2\ep}(t)\Big)=
e^{t/\ep}\Big(e^{-tA}v_1+\ep^{1/2}\widetilde{V}_{2\ep}(t)\Big).
\]
Integrating it over $[0,t]$, we can deduce
\begin{align*}
u_{2\ep}(t)= 
\int_0^t
e^{(s-t)/\ep}e^{-sA}v_1\,ds
+
\ep^{1/2}
\int_0^t
e^{(s-t)/\ep}
\widetilde{V}_{2\ep}(s)\,ds.
\end{align*}
For the first term, 
integration by parts yields that 
\begin{align*}
\int_0^t
e^{(s-t)/\ep}e^{-sA}v_1\,ds
&=
\left[
\ep e^{(s-t)/\ep}e^{-sA}v_1
\right]_{s=0}^{s=t}
+
\ep \int_0^t
e^{(s-t)/\ep}Ae^{-sA}v_1\,ds
\\
&=
\ep e^{-tA}v_1
-
\ep e^{-t/\ep}v_1
+
\ep \int_0^t
e^{(s-t)/\ep}Ae^{-sA}v_1\,ds,
\end{align*}
and then by using Lemma \ref{lem:max-reg} with $n=0$, one has
\begin{align*}
\left\|
\ep \int_0^t
e^{(s-t)/\ep}Ae^{-sA}v_1\,ds
\right\|
&\leq 
\ep \int_0^t
e^{(s-t)/\ep}\|Ae^{-sA}v_1\|\,ds
\\
&\leq 
\ep 
\left(\int_0^t
e^{2(s-t)/\ep}\,ds\right)^{1/2}
\left(
\int_0^t\|Ae^{-sA}v_1\|^2\,ds\right)^{1/2}
\\
&\leq 
\frac{\ep^{3/2}}{2} 
\|A^{1/2}v_1\|.
\end{align*}
Because of \eqref{eq:u2ep-error}, since one can obtain the following estimate  
\begin{align*}
\left\|
\ep^{1/2}
\int_0^t
e^{(s-t)/\ep}
\widetilde{V}_{2\ep}(s)\,ds
\right\|\leq 
C'
\ep^{1/2}\|A^{1/2}v_1\|
\int_0^t
e^{(s-t)/\ep}
\,ds
\leq 
C'
\ep^{3/2}\|A^{1/2}v_1\|,
\end{align*}
one has 
\begin{equation}\label{asymptotics2}
\|u_{2\ep}(t)-\ep(e^{-tA}v_1-e^{-t/\ep}v_1)\|\leq C''\ep^{3/2}\|A^{1/2}v_1\|.
\end{equation}
Combining \eqref{asymptotics1} and \eqref{asymptotics2}, 
one has arrived at the desired estimate.
\end{proof}

\subsection*{Acknowedgements}
The work of the first author is supported in part by Grant-in-Aid for Scientific Research (C) 15K04958 of JSPS. The work of the second author is partially supported 
by Grant-in-Aid for Scientific Research JP18K134450 of JSPS.


\end{document}